\documentclass{amsart}
\usepackage{amsmath,amsfonts,amsthm,amssymb,indentfirst,epic,xurl,graphics,needspace}

\newtheorem{theorem}{Theorem}
\newtheorem{lemma}[theorem]{Lemma}

\newtheorem{proposition}[theorem]{Proposition}
\newtheorem{definition}[theorem]{Definition}

\newtheorem{convention}[theorem]{Convention}

\newcommand{\R}{\mathbb{R}}
\newcommand{\cC}{\mathcal{C}}
\newcommand{\cP}{\mathcal{P}}
\renewcommand{\r}{\mathrm}

\begin{document}

\begin{center}
\texttt{Comments, corrections,
and related references welcomed, as always!}\\[.5em]
{\TeX}ed \today
\vspace{2em}
\end{center}

\title%
{A type of algebraic structure related to sets of intervals}
\thanks{%
Readable at \url{http://math.berkeley.edu/~gbergman/papers/}.
}

\subjclass[2010]{Primary: 06A05, 08A05, 08A55.
Secondary: 03C20, 03E25, 05C25, 05C62.
}
\keywords{total ordering on a set making a given family of subsets
convex; interval graph}

\author{George M. Bergman}
\address{University of California\\
Berkeley, CA 94720-3840, USA}
\email{gbergman@math.berkeley.edu}

\begin{abstract}
F.~Wehrung has asked: Given a family $\cC$ of subsets of
a set $\Omega,$ under what conditions will there exist
a total ordering on $\Omega$ with respect to which every member
of $\cC$ is convex?

We look at the family $\cP$ of
subsets of $\Omega$ generated by $\cC$ under certain
partial operations which preserve
convexity; we determine the possible structures of $\cP$
if $\cC,$ and hence $\cP,$ is finite, and note a condition
on that structure that is necessary and sufficient
for there to exist an ordering of $\Omega$ of the desired sort.
From this we obtain a criterion which works without the
finiteness hypothesis on $\cC.$

We establish bounds on the cardinality of
the set $\cP$ generated by an $\!n\!$-element set $\cC.$

We end by noting some other ways of answering Wehrung's
question, using results in the literature.
\end{abstract}
\maketitle

\section{Introduction}\label{S.intro}

If $(\Omega,\leq)$ is a totally ordered set, we shall call a subset
$A$ of $\Omega$ {\em convex} if for
$p,\,q,\,r\in\Omega$ with $p\leq q\leq r,$ the conditions
$p,\,r\in A$ imply $q\in A.$
(I use the word ``interval'' in the title, since
it concisely suggests ``convex subset of a totally ordered set'';
but in general, I shall write ``convex set''.
The word ``interval'' will, however, come up in \S\ref{S.int_grph},
in connection with results in the literature.)

This note answers the question:
\begin{equation}\begin{minipage}[c]{35pc}\label{d.FW_q}
(F.\,Wehrung, personal correspondence related to~\cite{FW_r-l};
slightly reworded.)
Let~$\Omega$ be a set (usually finite), and
$\cC$ a set of subsets of $\Omega.$
When does there exist a total order on~$\Omega$ with respect to
which every member of~$\cC$ is convex?
\end{minipage}\end{equation}

In \S\ref{S.inf<fin}, we will show that given a subset
$\cC$ of $\Omega,$ there exists a total ordering $\leq$
on $\Omega$ having the
desired property for $\cC$ if and only if for every finite subset
$\cC'\subseteq\cC,$ there exists a total
ordering $\leq_{\cC'}$ on $\Omega$ having that property for $\cC'.$
Thus, the general problem reduces to the
corresponding problem for finite $\cC.$

In \S\ref{S.patchw} we note some natural partial operations on the
set of convex subsets of a totally ordered set $\Omega,$
which can be described solely in set-theoretic
terms, and we call a set $\cP$ of
subsets of an arbitrary set $\Omega$ which is closed under
these partial operations a ``patchwork''.
Thus, a set $\cC$ of subsets of $\Omega$ has
all its members convex under some total
ordering $\leq$ on $\Omega$ if and only if the patchwork $\cP$
that it generates has the same property.
Moreover, the patchwork generated by a finite set $\cC$
of subsets of $\Omega$ is contained in the Boolean ring
of sets generated by $\cC,$ hence is again finite;
so question~\eqref{d.FW_q} for finite sets $\cC$
comes down to the same question for finite patchworks $\cP.$

In \S\ref{S.fin_cP} we determine the structures of all
finite patchworks $\cP,$ and note for which of these
there exist orderings of $\Omega$ making all members of $\cP$ convex.
In \S\ref{S.nonfin}, we combine this result with that
of \S\ref{S.inf<fin} to
get a condition on a not necessarily finite patchwork
which is necessary and sufficient for the existence of such an ordering.

In \S\ref{S.size} we examine how many elements the
patchwork $\cP$ generated by an $\!n\!$-element set $\cC$ can have.

After an earlier draft of this note was sent
out, it was pointed out to me that there
are results in the literature that can be used to answer
question~\eqref{d.FW_q} in other ways.
These are noted in~\S\ref{S.int_grph}.
So perhaps the main value of this paper is the detailed
structure it reveals of the families $\cP$ arising from families $\cC$
for which question~\eqref{d.FW_q} has a positive answer -- and
possibly also the corresponding information when
question~\eqref{d.FW_q} is not assumed to have a positive answer;
whether this has interesting applications I do not know.

We remark that the arguments of \S\ref{S.fin_cP} that give
the structures of all finite patchworks
are rather lengthy and intricate; but the reasoning is elementary.
In contrast, the argument below that reduces
question~\eqref{d.FW_q} to the finite case
calls on the Compactness Theorem of model theory,
or, alternatively, on ultraproducts.
So that proof, though brief, is the
one non-elementary piece of reasoning in this note.

\section{Reduction to the case of finite $\cC$}\label{S.inf<fin}

\begin{lemma}\label{L.inf<fin}
Let $\Omega$ be a set, and $\cC$ a set of subsets of $\Omega.$
Then the following conditions are equivalent.\\[.2em]
\textup{(i)} \ \ There exists a total order on~$\Omega$ with respect to
which every member of~$\cC$ is convex.\\[.2em]
\textup{(ii)} \ For each finite subset $\cC'\subseteq\cC,$
there exists a total order on~$\Omega$ with respect to
which every member of~$\cC'$ is convex.
\end{lemma}

\begin{proof}
(i)$\implies$(ii) is clear, since an ordering that works
for $\cC$ works for any subset $\cC'\subseteq\cC.$

The reader familiar with the Compactness Theorem of model theory
\cite[Corollary~5.6]{PMC.UA} will see that %
that theorem implies the reverse implication.
(One uses a language with a
constant for each element of $\Omega,$ a unary relation
for each member of $\cC,$ and a binary relation ${\leq}.)$
I sketch below a variant of this proof
that uses ultraproducts \cite[Theorem~5.1]{PMC.UA} rather than
the Compactness Theorem.

Assume~(ii).
Let $X$ be the set of all finite subsets
of $\cC,$ for each $\cC'\in X$ let $J_{\cC'}\subseteq X$ be the
set of all $\cC''\in X$ which contain $\cC',$
and let $\mathcal{F}$ be the set of all sets $J\subseteq X$
which for some $\cC'\in X$ contain $J_\cC'.$
Since $J_{\cC'}\cap J_{\cC''} = J_{\cC'\cup\,\cC''},$ we see that
$\mathcal{F}$ is a filter on $X,$ and since no $J_{\cC'}$ is empty,
$\mathcal{F}$ does not contain the empty set,
i.e., it is a proper filter.

Hence we can choose an ultrafilter $\mathcal{U}$ on $X$
containing $\mathcal{F}.$
By~(ii), we can choose for each $\cC'\in X$ a total ordering
$\leq_{\cC'}$ on $\Omega$ with respect to which all
members of $\cC'$ are convex.
The ultraproduct of the totally ordered sets
$(\Omega,\,\leq_{\cC'})$ $(\cC'\in X)$ with respect to $\mathcal{U}$
will be a totally ordered set $(\Omega^*,\,\leq^*),$
such that for each
$A\in\cC,$ the ultrapower of $A$ with respect to $\mathcal{U}$
is a convex subset $A^*$ of $\Omega^*.$
$(A^*$ is convex because $A$ is convex with respect to ``almost all''
the chosen orderings~$\leq_{\cC'};$
i.e., the set of $\cC'$ such that $A$ is convex with
respect to $\leq_{\cC'}$ contains $J_{\cC'},$ hence
belongs to $\mathcal{U}.)$

There is a natural embedding of $\Omega$ in
its ultrapower $\Omega^*,$ and for each $A\in\cC,$
the inverse image in $\Omega$ of $A^*\subseteq\Omega^*$ is $A.$
Hence the restriction of $\leq^*$ to $\Omega$ is a
total ordering $\leq$ under which each $A\in\cC$ is convex, proving~(i).
\end{proof}

\section{Patchworks}\label{S.patchw}

The kind of structure we will call a patchwork is motivated by

\begin{lemma}\label{L.ops}
If $(\Omega,\,\leq)$ is a totally ordered set, and $A,~B$
are convex subsets of $\Omega$ which have nonempty intersection, and
neither of which contains the other, then $A\cap B,$
$A\cup B,$ and $A\setminus B$ are also convex.

Also, $\emptyset$ and $\Omega$ are convex.
\qed
\end{lemma}

(Each of the conclusions of the first sentence of Lemma~\ref{L.ops}
needs less than the full set of assumptions on $A$ and $B.$
But the cases where the unnecessary assumptions fail are
trivial, so we lose nothing in using this formulation.)

\begin{definition}\label{D.patchw}
Let $\Omega$ be a set.

We will say that subsets $A,\,B\subseteq\Omega$ {\em overlap}
if they have nonempty intersection, but neither contains the other.

For any set $\Omega,$ a {\em patchwork} of subsets of $\Omega$
\textup{(}which may be shortened to ``a patchwork on $\Omega$'' or
simply ``a patchwork'' when the context makes $\Omega$ clear\textup{)}
will mean a set $\cP$ of subsets of $\Omega$ such that
for every pair of {\em overlapping} sets
$A,\,B\in\cP,$ the set $\cP$ also contains $A\cap B,$
$A\cup B,$ and $A\setminus B,$ and such that
$\emptyset$ and $\Omega$ also belong to $\cP.$

If $\cP$ is a patchwork on $\Omega,$
members of $\cP$ will be called {\em $\!\cP\!$-sets}.
\end{definition}

We note

\begin{lemma}\label{L.finite}
The patchwork $\cP$ on a set $\Omega$ generated by a
finite set $\cC$ of subsets of $\Omega$ is finite.
\end{lemma}

\begin{proof}
$\cP$ will be contained in the Boolean ring of subsets of $\Omega$
generated by $\cC,$ and a finitely generated Boolean ring is finite.
\end{proof}

\section{The case of finite $\cP.$}\label{S.fin_cP}

\begin{convention}\label{C.cP_fin}
Throughout this section $\Omega$ will be a fixed
set, and $\cP$ a {\em finite} patchwork on $\Omega.$
\end{convention}

A key to analyzing the structure of $\cP$ is
to look at those of its members to which its
partial operations {\em cannot} be applied:

\begin{definition}\label{D.aut}
A $\!\cP\!$-set $A$ will be called {\em autonomous}
if it is nonempty, and does not overlap any member of~$\cP$
\textup{(}cf.~Definition~\ref{D.patchw}\textup{)}.
We will generally shorten ``autonomous $\!\cP\!$-set'' to
``autonomous set''.

If $A\in\cP$ is autonomous,
then the set of maximal autonomous {\em proper} subsets of $A$
will be called the {\em cohort} under $A.$
Elements of $A$ that are not in the union of the cohort
under $A$ will be called {\em non-cohort elements} of $A.$
\end{definition}

Note that

\begin{lemma}\label{L.2aut}
Any two autonomous members of a patchwork $\cP$
are either disjoint, or one contains the other.
\end{lemma}

\begin{proof}
By the definition of ``autonomous'' this is, in fact, true of any
two $\!\cP\!$-sets at least {\em one} of which is autonomous.
\end{proof}

Thus the autonomous members of $\cP$
form a tree under inclusion, branching
downward from the top element, $\Omega.$
It follows that for every proper subset $A$ of $\Omega,$
there is a least autonomous $\!\cP\!$-set $B$ properly containing~$A.$

Note that every {\em minimal}
autonomous set $A$ consists entirely of non-cohort elements.
Non-minimal autonomous $\!\cP\!$-sets necessarily have nonempty cohorts
under them; we shall see that
they may or may not also have non-cohort elements.
Each element of $\Omega$ is a non-cohort element of a
unique autonomous set, namely, the least autonomous
set containing it.

Not every union of autonomous sets need be a $\!\cP\!$-set, but
the converse is true:

\begin{lemma}\label{L.cup_aut}
Let $A$ be a $\!\cP\!$-set.
Then\\[.2em]
\textup{(i)} \ \ $A$ is a union of autonomous sets; in fact
it is the disjoint union of the maximal
elements of the set of autonomous sets which it contains.

Moreover, if $A\neq\Omega,$ then letting $B$ be the least
autonomous set {\em properly} containing $A,$ we have\\[.2em]
\textup{(ii)} \ The maximal autonomous sets contained in $A$
\textup{(}cf.~\textup{(i))}
all belong to the cohort under $B.$\\[.2em]
\textup{(iii)} If $A'$ is a $\!\cP\!$-set which overlaps
$A,$ then the least
autonomous set properly containing $A'$ is also $B.$\\[.2em]
\textup{(iv)} If $A$ is not autonomous, then
it can be written as the disjoint union of two proper $\!\cP\!$-subsets,
each of which is a union of subsets of the cohort under $B.$
\end{lemma}

\begin{proof}
(i):  If $A$ is autonomous this is
trivial:\ it is the union of the $\!1\!$-element family $\{A\}.$
If $A$ is not autonomous,
assume inductively that all $\!\cP\!$-sets properly contained in $A$
are unions of autonomous sets.
Since $A$ is not autonomous, it overlaps some $\!\cP\!$-set $A'.$
Then $A\cap A'$ and $A\setminus A'$ are $\!\cP\!$-sets
properly contained in $A,$ so by inductive hypothesis,
they are unions of autonomous sets, hence so is their union,~$A.$
That $A$ is the union of the {\em maximal} autonomous sets
that it contains follows, since $\cP$ is finite.
These sets are disjoint by Lemma~\ref{L.2aut}.

(ii):  If $A$ is autonomous, then $A$ is
the unique maximal autonomous subset of $A,$
and the conclusion is immediate.

If $A$ is not autonomous, suppose, by way of contradiction,
that some maximal autonomous subset $C$ of $A$ belonged to a different
cohort from the cohort under $B;$ say the cohort
under an autonomous set $D.$
This makes $D$ the least autonomous set properly containing $C;$
but $C\subset A \subset B,$ so $B$  is an
autonomous set properly containing $C,$ so $D\subset B,$
so as $B$ is the least autonomous subset containing $A,$
$A$ is not contained in $D.$
Neither can $A$ contain $D;$ if it did, $C$ would
not be a {\em maximal} autonomous subset of $A.$
So $A$ overlaps $D;$ but
this contradicts the assumption that $D$ is autonomous,
completing the proof.

(iii): We shall show that $B$ contains $A'.$
Then, by the same argument with the roles
of $A$ and $A'$ reversed, the least autonomous set $B'$
properly containing $A'$ contains $A.$
Thus, $B$ and $B'$ must be the same, giving the desired assertion.

Note that $A\cup A'$ is a $\!\cP\!$-set which is
not disjoint from $B,$ so by the autonomy of $B,$
it must either be contained in $B$ or properly contain $B.$
If the latter were true, then $A'$ would have elements outside
$B,$ and also contain the
nonempty set $B\setminus A;$ but since it does
not contain all of $A\subset B,$ it would not contain $B,$
contradicting the autonomy of $B.$
So $A\cup A'$ must be contained in $B,$
so $A'$ is contained in $B,$ as required.

(iv): Since $A$ is nonautonomous, we can choose a
$\!\cP\!$-set $C$ which it overlaps.
Thus, by~(iii), the least autonomous set containing $C$ is $B,$
so by~(i) and~(ii), $A$ and $C$ are both unions
of members of the cohort under $B.$
Now $A$ is the union of the disjoint nonempty $\!\cP\!$-sets
$A\cap C$ and $A\setminus C;$
so this is a decomposition of desired sort.
\end{proof}

The above results suggest
the question: Given an autonomous set with nonempty cohort under it,
{\em which} unions of subsets of this cohort can be $\!\cP\!$-sets?
We will use the following concept.
(We define it here for general $\!\cP\!$-sets,
but until the last result of this section,
we will only use it for autonomous sets.)

\begin{definition}\label{D.adj}
Two $\!\cP\!$-sets $A$ and $B$ will be called {\em adjacent}
if they are disjoint, and their union is again a $\!\cP\!$-set.

The {\em adjacency graph} of a family of $\!\cP\!$-sets
will mean the graph having the members of this family as
vertices, and having an edge between a pair of such
vertices if and only if they are adjacent as $\!\cP\!$-sets.
\end{definition}

The word ``adjacent'' is, of course, motivated by the case where
$\cP$ consists of sets convex under a total ordering $\leq$ on $\Omega.$
But our definition is not limited to that case;
and even in that case, two sets $A,\,B\in\cP$
that are adjacent as convex subsets of $\Omega$ need not be
adjacent under our present definition, if our patchwork
$\cP$ doesn't happen to include $A\cup B.$

\begin{lemma}\label{L.adj-graph}
A subset $A$ of $\Omega$ is a $\!\cP\!$-set if and only
if it can be written as the union of a family of
autonomous sets whose adjacency graph is connected.
\end{lemma}

\begin{proof}
The statement is trivial if $A$ is an autonomous
set or $\emptyset,$ so assume that neither of these is the case.

To get the ``if'' direction, suppose $A$ is the union of
a family of more than one autonomous sets which has
connected adjointness graph.
Note that the implication in this direction is, in fact, also
trivially true if the graph in question has exactly $2$ vertices,
by the definition of adjacency;
so let us assume it has $\geq 3$ vertices.
Let us also assume inductively that the corresponding implication
is true for every set $A'$ that can be
written as the union of a family of fewer
autonomous sets than we are using for~$A.$

Let us now choose a vertex $B$ in the graph for $A$
such that the subgraph obtained by removing $B,$ i.e.,
the adjacency graph of the family with union $A\setminus B,$
is still connected.
(Such a vertex exists by~\cite[Theorem~3.2.10]{graph}.)
Thus, by our inductive assumption, $A\setminus B$ is a $\!\cP\!$-set.
Let $C$ be a vertex of the graph for $A$ which is adjacent to $B.$
Since the graph for $A$ has $\geq 3$ vertices,
it has vertices other than $B$ and $C.$
Hence neither of the $\!\cP\!$-sets $A\setminus B$ and $B\cup C$
contains the other, moreover, they overlap in $C;$ so their union,
$A,$ is a $\!\cP\!$-set, as desired.

Conversely, let $A$ be a $\!\cP\!$-set which is neither autonomous
nor empty, and $B$ be the least autonomous set containing $A.$
By Lemma~\ref{L.cup_aut}(i) and~(ii), $A$ is the union
of a subset of the cohort under $B;$ we wish to show that
the adjacency graph of this expression for $A$ is connected.
Again, this is clear if that graph has $\leq 2$ vertices,
so suppose it has at least $3,$ and assume by induction
that every subfamily of the cohort under $B$ whose union is
a proper $\!\cP\!$-subset of $A$ corresponds to a connected
subgraph of that graph.

Since $A$ is non-autonomous,
Lemma~\ref{L.cup_aut}(iv) gives us a decomposition of our
expression for $A$ into two disjoint subfamilies, corresponding
to disjoint $\!\cP\!$-sets $A_1$ and $A_2$ with union $A.$
By our inductive assumption, the graphs corresponding
to these $\!\cP\!$-sets are connected, so to show that the
graph corresponding to $A$ is connected, it will suffice
to find a connected subgraph thereof
that meets both the indicated subgraphs.
Again invoking our inductive assumption, we will have this if
we can find a $\!\cP\!$-set $A_3\subset A$
which is a union of members of the cohort under $B,$
and which meets both $A_1$ and $A_2.$

Since the graph associated with $A$ had at least $3$ vertices,
one of the two subgraphs we have obtained must have more than one
vertex; i.e., one of our two $\!\cP\!$-subsets,
say $A_1,$ must be non-autonomous.
Let $C$ be a $\!\cP\!$-set which overlaps $A_1.$
By Lemma~\ref{L.cup_aut}(iii), $C$ is also a union of
members of the cohort under $B.$
If $C$ also has nonempty intersection with $A_2,$
then we see that $A\cap C$ will be a $\!\cP\!$-subset $A_3$
of the sort we need.
On the other hand, if $C\cap A_2=\emptyset,$ then the
fact that $C\not\subseteq A_1$ implies $C\not\subseteq A,$
so $A\setminus C$ is a $\!\cP\!$-set, and will serve as our $A_3.$
\end{proof}

(Remark:  In in the above result, if $A$ is not autonomous, then
it will have a {\em unique} expression of the indicated sort.
If it {\em is} autonomous, then as noted, the expression as the
union of the singleton $\{A\}$ is always of the indicated sort;
in this case $A$ may or may not have another such expression.
Namely, we will see later that an autonomous
set $A$ may or may not be the union of
the cohort under it, and that if it is, the adjacency graph
for that decomposition may or may not be connected.)

To lead up to our criterion for the existence of an
ordering of $\Omega$ making all $\!\cP\!$-sets convex,
we look next at two conditions which (it is not hard to see)
are incompatible with the existence of such an ordering.

\begin{lemma}\label{L.adj_to_3}
If an autonomous set $A$ is adjacent to {\em more than} two
other autonomous sets, then every two autonomous
sets adjacent to $A$ are adjacent to one another.

In fact, when this is true, then
in the adjacency graph of the cohort containing $A,$
the connected component of $A$ has an edge between every two vertices.
\end{lemma}

\begin{proof}
Let $B$ be the least autonomous set properly containing $A;$
thus, the lemma concerns the cohort under~$B.$

To get the first assertion, we first consider any
two autonomous sets $C\neq C'$ adjacent to $A.$
By the adjacency assumption, $A\cup C$ and $A\cup C'$
are $\!\cP\!$-sets, and they clearly overlap,
so their union $A\cup C\cup C'$ is a $\!\cP\!$-set.
Now by our hypothesis there is at least one other autonomous
set $C''$ adjacent to $A,$ so $A\cup C''$ is a $\!\cP\!$-set
which overlaps $A\cup C\cup C';$
hence $(A\cup C\cup C')\setminus (A\cup C'') = C\cup C'$
is a $\!\cP\!$-set, i.e., $C$ and $C'$ are adjacent, as claimed.

To get the second assertion, let us first show that
in the situation of the above paragraph, if $C$
is adjacent to yet another member $D$ of the cohort under $B,$
then $A$ is also adjacent to $D.$
Indeed, since $C$ is adjacent to the three
sets $A,$ $C'$ and $D,$ the first assertion of the lemma shows
that these are adjacent to one another,
in particular, $A$ and $D$ are adjacent.
By induction, all members of the connected component in question
are adjacent to $A,$
hence by yet another application of the first assertion
of the lemma, they are all adjacent to one another.
\end{proof}

\begin{lemma}\label{L.cycle}
Suppose $A_1,\dots,A_n$ $(n\geq 3)$ are distinct autonomous sets
forming a ``cycle of adjacency'', i.e.,
such that $A_i$ is adjacent to $A_{i+1}$ $(1\leq i<n),$
and $A_n$ to $A_1.$
Then, again, in the adjacency graph of the cohort containing
the $A_i,$ the connected component containing
these elements has an edge between every two vertices.
\end{lemma}

\begin{proof}
First suppose $n\geq 4.$
Then $A_1\cup A_2\cup A_3$ and
$A_3\cup\dots\cup A_n\cup A_1$ will be overlapping $\!\cP\!$-sets,
hence their intersection, $A_1\cup A_3$ is a $\!\cP\!$-set; i.e.,
$A_1$ and $A_3$ are adjacent.
Hence $A_3$ is adjacent to at least the three sets
$A_1,$ $A_2$ and $A_4,$ so the preceding lemma is applicable.

On the other hand, suppose $n=3.$
If the connected component in question
is simply $\{A_1,\,A_2,\,A_3\},$ we are done.
If not, then by connectedness,
some other member $C$ of that connected component
must be adjacent to one of the $A_i.$
This again makes that $A_i$ adjacent to three distinct
members of the component, so again, the preceding lemma is applicable.
\end{proof}

We can now prove

\begin{theorem}\label{T.i-iii}
Let $\cP$ be a finite patchwork, and $B\in\cP$ an autonomous set.
Then exactly one of the following statements is true:\\[.2em]
\textup{(i)}~ The cohort under $B$ has at least $3$ members,
every two members of that cohort are adjacent,
and the union of the cohort is all of $B.$ \\[.2em]
\textup{(ii)}\, The cohort under $B$ has at least $2$ members,
the adjacency graph of that cohort is a path
\textup{(}i.e., consists of elements $A_1,\dots,A_n$ for some $n\geq 2,$
with $A_i$ and $A_j$ adjacent if and only if $|i-j|=1),$
and again, the union of the cohort is all of $B.$\\[.2em]
\textup{(iii)}  No two members of the cohort under $B$ are adjacent.
\textup{(}In this case, $B$ may contain non-cohort elements.
If it does,
the cohort under $B$ may be empty or nonempty.\textup{)}\vspace{.2em}
\end{theorem}

\begin{proof}
Consider the adjacency graph of the cohort under $B.$
If this has any edges, then the members of that cohort
comprising a nontrivial connected component of that
graph will have as union a $\!\cP\!$-subset $A\subseteq B$
by Lemma~\ref{L.adj-graph}.
If this $A$ were not autonomous, then Lemma~\ref{L.cup_aut}(iii)
would give a $\!\cP\!$-set $A'$ overlapping it,
which by parts~(ii) and~(i) of that lemma would also
be a union of members of the cohort under $B;$ hence
$A\cup A'$ would be a larger $\!\cP\!$-subset of $B,$ and by
Lemma~\ref{L.adj-graph} the adjacency graph of that larger
$\!\cP\!$-set would be a larger connected subgraph of the adjacency
graph of $B,$ a contradiction.
So $A$ is autonomous.

But as an autonomous $\!\cP\!$-subset of $B$ which
properly contains members of the cohort under $B,$ $A$ must equal $B.$
So we conclude that if the adjacency graph of the cohort under $B$ has
any edges, then it is connected and $B$ is the union of
that cohort; in particular, $B$ has no non-cohort elements.

By Lemmas~\ref{L.adj_to_3} and~\ref{L.cycle}, if the above
adjacency graph has either any ``branching'' or any cycles,
it must have an edge between every pair of points, and since
it will also have at least three vertices, it falls under case~(i).
On the other hand, a connected graph with at least one edge
which has neither branching nor cycles is a path, so if that is so,
we are in case~(ii).
Finally, if the adjacency graph of $B$ has no edges, we are
in case~(iii).

(As noted earlier, a minimal autonomous set $B$ has empty cohort under
it, and so necessarily has non-cohort elements.
For examples of autonomous sets having both a nonempty cohort under
it and
non-cohort elements, the reader can examine the case where $\Omega$ is a
totally ordered set, and $\cP$ is the patchwork
generated by an interval $B$
in $\Omega$ and one or more pairwise disjoint subintervals of $B$
whose union is not all of $B.)$
\end{proof}

We can now answer question~\eqref{d.FW_q} for finite patchworks.

\begin{theorem}\label{T.cnvx}
The following three conditions on a finite patchwork
$\cP$ are equivalent:\\[0.2em]
\textup{(a)} $\Omega$ admits a total ordering $\leq$
under which all $\!\cP\!$-sets are convex.\\[0.2em]
\textup{(b)} No three nonempty $\!\cP\!$-sets
are pairwise adjacent.\\[0.2em]
\textup{(c)} Every autonomous set $B\subseteq\Omega$ falls under
case~\textup{(ii)} or \textup{(iii)} of
Theorem~\ref{T.i-iii}.
\end{theorem}

\begin{proof}
(a)$\implies$(b) is intuitively clear.
To supply some details: if a convex set $A$ is adjacent to
each of a pair of disjoint convex sets $A'$ and $A'',$ then one
of $A'$ and $A''$ must lie above it and the other below it under $\leq,$
otherwise $A'$ and $A''$ could not be disjoint.
Since $A'$ and $A''$ have the nonempty
set $A$ between them, their union is not
convex, so they are not adjacent.

(b)$\implies$(c) is immediate, since if some $B$
fell under case~(i) of Theorem~\ref{T.i-iii}, the
cohort under it would include three pairwise adjacent sets.

To prove (c)$\implies$(a) we shall, assuming~(c),
construct recursively an ordering $\leq$ as in~(a).

Given an autonomous set $B\subseteq\Omega,$ assume recursively
that for each $A$ in the cohort under $B,$ we have constructed
a total ordering on $A$ under which all $\!\cP\!$-sets contained
in $A$ are convex.
(If $B$ is minimal, the cohort under it is empty, and this assumption is
vacuous.)

If $B$ falls under case~(ii) of Theorem~\ref{T.i-iii},
then writing the members of the cohort under $B$ as
$A_1,\dots,A_n,$ as in that theorem, we order their union
$B$ so that $A_1,\dots,A_n$ form {\em successive} intervals of $B$
(either putting $A_i$ below $A_{i+1}$ for all $i,$ or
putting $A_i$ above $A_{i+1}$ for all $i),$
with each ordered internally
by the order previously constructed for it.

On the other hand, if $B$ falls under case~(iii),
we can order it by arranging the members of the cohort under
it (if any)
in any order, simply making each a convex subset of $B,$ again
with the internal ordering previously constructed,
and likewise put the non-cohort elements of $B$
(if any) in any order above, below or between those
intervals (but not, of course, within any of them).

We see that for every $\!\cP\!$-set $A,$ if $B$ is the least
autonomous $\!\cP\!$-set containing it,
then $A$ becomes convex under the ordering so constructed on $B,$
and remains so as we extend this ordering to larger autonomous sets.
Hence when our construction reaches
the top set, $\Omega,$ we have established~(a).
\end{proof}

Looking back at Theorem~\ref{T.i-iii}, we can turn this upside down, and
describe, somewhat informally, how to construct all finite patchworks:

Start with a finite partially ordered set $T$ having the form of
a downward-branching tree, with greatest element denoted $\Omega.$
Choose any graph structure on the elements of $T$
such that edges, if any, occur only among elements of $T$
that lie immediately below a common element, and such that for
each element, the resulting
graph structure on the elements immediately
below it is either a complete graph, a path, or edgeless.
Now assign to each element of $T$ such that the graph
structure we have given to the set of elements immediately below it
is edgeless a set that is to be its set of non-cohort
elements, assigning to distinct elements of $T$ disjoint sets,
and making the set so assigned to every such
element of $T$ having $\leq 1$ elements of $T$ immediately
below it nonempty.
Let each element of $T$ become the name of the union
of the sets of non-cohort elements assigned to it and
to elements of $T$ {\em anywhere below it}, and call the
sets so named {\em autonomous sets}.
(Apologies for this informal step of turning
abstract elements into names for sets!)
In particular, $\Omega$ is the name assigned to the set of
{\em all} elements that have been introduced as non-cohort elements.
Finally, let the members of $\cP$ be the unions of those families
of autonomous sets which, under our graph structure on $T,$
form connected subgraphs.

\section{The case of not-necessarily-finite $\cP.$}\label{S.nonfin}

Returning to Theorem~\ref{T.cnvx}, we can combine it with
Lemma~\ref{L.inf<fin} and deduce:

\begin{theorem}\label{T.cnvx-inf}
Let $\Omega$ be a set, and $\cP$ a
\textup{(}not necessarily finite\textup{)} patchwork
of subsets of $\Omega.$
Then \textup{(}as in Theorem~\ref{T.cnvx}\,\textup{)},
the following conditions are equivalent.\\[.2em]
\textup{(a)} $\Omega$ admits a total ordering $\leq$
under which all $\!\cP\!$-sets are convex.\\[0.2em]
\textup{(b)} No three nonempty $\!\cP\!$-sets are pairwise adjacent.
\end{theorem}

\begin{proof}
(a)$\implies$(b) is clear, as in the proof of Theorem~\ref{T.cnvx}.

Conversely, assume~(b) holds.
Then the same condition holds for all finite
sub-patchworks $\cP'\subseteq\cP,$ hence by
Theorem~\ref{T.cnvx}, for each such $\cP'$ there exists
an ordering $\leq_{\cP'}$ making all elements of $\cP'$ convex.
Since every finite subset $\cC'\subseteq\cP$ generates a finite
sub-patchwork $\cP',$ an ordering $\leq_{\cP'}$ which makes
all members of $\cP'$ convex will do the same for
$\cC';$ so we can now apply
Lemma~\ref{L.inf<fin} (with $\cP$ in the role of $\cC)$
to get an ordering $\leq$ as in~(a), completing the proof.
\end{proof}

We could not bring condition~(c) of Theorem~\ref{T.cnvx}
into Theorem~\ref{T.cnvx-inf}, because
as one goes from one finite patchwork $\cP_0$
on $\Omega$ to a larger one, sets that
were autonomous can cease to be so, and this may leave us with no
proper autonomous sets in the infinite set $\cP.$
For two contrasting examples, let $\Omega$ be the real line,
let $\cP_1$ consist of $\emptyset,$ $\Omega,$ and all finite
half-open intervals
$[a,\,b),$
and let $\cP_2$ consist of all finite unions of members of $\cP_1.$
Both these families are patchworks, and it is easy to see
that neither has any autonomous elements other than $\Omega.$
The sets in $\cP_1$ are convex under
the standard ordering on $\Omega,$ but
there can be no ordering under which the sets in $\cP_2$
are convex, since
condition~(b) fails: {\em every} two disjoint nonempty
members of $\cP_2$ are adjacent in that family, so any three such
sets show the failure of~(b).
\vspace{.5em}

It would be interesting to know whether the kind of structures
we have called patchworks are useful in other contexts
than the study of convex subsets of ordered sets; in
particular, whether patchworks that do not satisfy the conditions
for convexity under an ordering on $\Omega$ occur in any natural way.

We remark that in defining patchworks, rather than making them
families of subsets of a set $\Omega,$ we could, with slightly
greater formal generality, have made them families of subsets
of a Boolean ring.
The ``subsets of a set'' definition seemed simplest for our
purposes, but subsets of a Boolean ring might be better in another.

\section{The cardinality of a patchwork generated by $n$ elements}\label{S.size}
We noted in Lemma~\ref{L.finite} that every finitely generated
patchwork is finite.
Let us get some more precise bounds.

The proof of the next
result uses an observation that we have not yet
stated explicitly: any patchwork $\cP$ is closed under
pairwise intersections, and hence under all finite intersections.
Indeed, by definition $\cP$ is closed under intersections of
overlapping elements, while if two elements are comparable,
their intersection is one of them, and if they are disjoint,
their intersection is $\emptyset\in\cP.$
This gives us all pairwise intersections, and hence, as noted, all
finite intersections.

\begin{proposition}\label{P.size}
If $n$ is a nonnegative integer, $\Omega$ a set, and $\cP$
a patchwork generated by $n$ subsets of $\Omega,$ then $\cP$
has at most $2^{2^n-1}+1$ elements.
Moreover, for all $n\neq 2,$ there exist examples achieving this bound.
\end{proposition}

\begin{proof}
Let $\mathcal{B}(n)$ be the free Boolean ring on $n$ generators
\cite[Theorem~26, p.260, and Corollary~1]{frBA}.
This can be described as the Boolean ring of all subsets
of the set $2^n$ of subsets of $\{1,\dots,n\};$ its free
generators are the sets $X_1,\dots,X_n,$ where $X_i$ is
the set of those subsets containing $i.$
Thus $\mathcal{B}(n)$ has cardinality $2^{2^n}.$
The nonunital subring $\mathcal{B}(n)^\r{o}$
of $\mathcal{B}(n)$ generated by
the $X_i$ consists of those members
of $\mathcal{B}(n)$ which, like $X_1,\dots,X_n,$ do
not have as a member the empty subset of $\{1,\dots,n\};$
so it has cardinality $2^{2^n-1}.$

Given any $n$ subsets $A_1,\dots,A_n$ of a set $\Omega,$
we can map the free Boolean
ring $\mathcal{B}(n)$ homomorphically into the Boolean ring of
all subsets of $\Omega$ so as to send each $X_i$ to $A_i.$
It is not hard to see that the partial binary operations
of Definition~\ref{D.patchw} carry the image of
$\mathcal{B}(n)^\r{o}$ into itself, so the elements
of $\cP$ obtainable using those operations will lie
in that image, as will the element $\emptyset.$
The one member of $\cP$ that may not be contained in
that image is $\Omega,$ so counting this,
we see that $\cP$ must have cardinality $\leq 2^{2^n-1}+1.$

To get an example where this value is achieved,
let us take $\Omega$ to be $2^n,$ defined as above as the set
of all subsets of $\{1,\dots,n\},$ and examine
the structure of the patchwork $\cP$ on $\Omega$ generated
by the $X_i.$

Since every pair of $X_i$ overlap, $\cP$ will contain
all the sets $X_i\setminus X_j$ $(i\neq j).$
I now claim that for any $a\in 2^n$
which contains some but not all of $1,\dots,n,$
the singleton $\{a\}$ is a $\!\cP\!$-set.
Indeed, it is the intersection of
those sets $X_i\setminus X_j$ such that $i\in a$ but $j\notin a.$
The singleton whose unique member is the improper subset,
$\{1,\dots,n\},$ is also in $\cP,$ being the intersection
of the $X_i$ themselves.
On the other hand, $\emptyset\in\Omega$
is not contained in any of the $X_i,$
so the singleton $\{\emptyset\}$ is not a $\!\cP\!$-set.

This gives us $2^n-1$ singleton autonomous sets; let us
now look at adjacency relationships among them.
I claim that if two elements $a\neq\,b\in\Omega$
are nonempty, and differ only in the presence or
absence of a single one of $1,\dots,n,$ then
$\{a,\,b\}\in\cP,$ i.e., $a$ and $b$ are adjacent.
Indeed, of the $n-1$ elements with respect to which $a$ and $b$ agree,
they must agree in {\em containing} at least one, otherwise one
of $a$ or $b$ would be empty.
If they also agree in {\em not} containing at least one, then,
imitating the trick of the preceding paragraph, we can obtain
$\{a,\,b\}$ as an appropriate intersection
of sets $X_i\setminus X_j,$ while if they contain
all $n-1$ elements at which they agree, then,
again using the idea of the preceding paragraph, we find that
$\{a,\,b\}$ will be an intersection of sets $X_i.$
Using these adjacency relations,
we can get a chain of adjacency from any
nonempty $a$ to the element $\{1,\dots,n\};$ so the
adjacency graph of these $2^n-1$ singleton autonomous
sets is connected.
Moreover, the $n$ elements of cardinality $n-1$ are all adjacent
to $\{1,\dots,n\},$ so if $n\geq 3,$ Lemma~\ref{L.adj_to_3} tells
us that every two of these singletons are adjacent; so the union of
every subfamily of these singletons is
a $\!\cP\!$-set, giving $2^{2^n-1}$ elements.
Bringing in $\Omega$ itself
(which is not such a union, because it contains
not only the $2^{n-1}$ nonempty subsets of $\{1,\dots,n\}$ but also
the empty subset), we have the asserted $2^{2^n-1}+1$ elements.

We assumed in the next-to-last sentence that $n\geq 3.$
Looking at lower values,
if $n=1,$ so that $\{1,\dots,n\}$ has a unique nonempty
subset $a,$ we have $2^{2^n-1}+1 = 3,$ and
$\cP$ indeed has $3$ elements, $\emptyset\subset\{a\}\subset\Omega.$
If $n=0,$ so that $\Omega=\{\{\emptyset\}\},$ the empty
set $\cC$ of subsets of $\Omega$ generates the patchwork
$\{\emptyset,\,\Omega\},$ which has cardinality $2 = 2^{2^n-1}+1.$

The case $n=2,$ on the other hand will fall under the next result.
(See discussion after the proof thereof.)
\end{proof}

\begin{proposition}\label{P.size_cnvx}
If $n$ is a nonnegative integer, $\Omega$ a set, and $\cP$
a patchwork generated by $n$ subsets of $\Omega$ whose members
are all convex under some ordering $\leq$ on $\Omega,$
then $\cP$ has at most $\binom{2n}{2} + 2 = 2n^2-n+2$ elements; and
for all $n$ there exist examples achieving this bound.
\end{proposition}

\begin{proof}
To prove the stated bound, suppose $(\Omega,\leq)$ is a totally ordered
set, and let us define a ``cut'' $\chi$ in $\Omega$ to be a pair
$(\chi^\ell,\,\chi^u)$ (the superscripts standing
for ``lower'' and ``upper'') such that $\Omega$ is the disjoint
union of $\chi^\ell$ and $\chi^u,$ and every element of $\chi^\ell$
is less than every element of $\chi^u.$
(So such a cut is equivalent to an isotone map
from $\Omega$ to $\{0,1\}.$
Note that $\chi^\ell$ or $\chi^u$ may be empty.)
Let us call a cut $\chi$ ``higher than'' a cut
$\chi'$ if $\chi^\ell\supset{\chi'}^\ell;$
clearly the cuts are totally ordered under this relation.

Observe that every nonempty convex subset $A$
of $\Omega$ is determined by two cuts, $\chi_{A\uparrow}$
and $\chi_{A\downarrow},$ the former characterized by the property that
$\chi_{A\uparrow}^u$ consists of all elements of $\Omega$ that are
greater than all elements of $A$ (equivalently, that
$\chi_{A\uparrow}^\ell$ consists of elements $\leq$ at least one
element of $A),$ the latter by the property that
$\chi_{A\downarrow}^\ell$ consists of the elements that are
less than all elements of $A$ (equivalently, that
$\chi_{A\downarrow}^u$ consists of the elements $\geq$ at least one
element of $A).$
Thus, $A=\chi_{A\uparrow}^\ell\cap \chi_{A\downarrow}^u.$

Given overlapping convex sets $A$ and $B,$ it is not hard to see
that for each of $A\cup B,$ $A\cap B,$ $A\setminus B$ and
$B\setminus A,$ the two cuts determining the set in question
are taken from among the
four cuts $\chi_{A\downarrow},$ $\chi_{A\uparrow},$
$\chi_{B\downarrow},$ and $\chi_{B\uparrow}.$
(Which pair of cuts correspond to each of these depends on whether
$A$ extends upward from $A\cap B$ and $B$ downward therefrom,
or vice versa.
For instance, in the former case, $A\setminus B$ is determined by
the two cuts $\chi_{A\uparrow}$ and $\chi_{B\uparrow}.)$

Hence if we start with a set $\cC$ of $n$ convex sets, those of them
that are nonempty determine at most $2n$ distinct cuts, and
applying our four partial binary operations recursively,
we get a set of nonempty convex sets each of which is determined by
one of the $\leq\binom{2n}{2}$ unordered pairs of distinct
cuts from this set.
Bringing in the sets $\Omega$ and $\emptyset$ (which may or
may not have belonged to $\cC,$ or, in the case of $\Omega,$
arisen via our partial operations), we see that the patchwork
generated by $\cC$ can have at most $\binom{2n}{2}+2$ elements.

Let us now show that for every $n,$ this upper bound can be achieved.
For integers $i\leq j,$ let us here write $[i,j]$ for
$\{k\in\mathbb{Z}\mid i\leq k\leq j\}.$
Let $\Omega$ be the set $[-n,\,n],$
under the usual ordering of the integers, and
let $\cC$ consist of the $n$ convex subsets
$A_i = [-n+i,\,i-1]$ $(i=1,\dots,n),$ each of which has $n$ elements.
We shall show below that the closure
$\cP$ of $\cC$ under the partial binary
operations in the definition of a patchwork consists of
{\em all} the nonempty convex subsets of $[-n+1,\,n-1].$
Note that the latter $\!2n{-}1\!$-element ordered set has $2n$ cuts (one
``below'' the least element, $2n-2$ between successive
elements, and one ``above'' the greatest element),
hence it has $\binom{2n}{2}$ nonempty convex subsets;
once we know that our partial binary operations give
us all these sets, then bringing in $\Omega$
(which, since it contains $-n$
and $n,$ is not in that closure) and $\emptyset,$
it will follow that the patchwork generated by $\cC$
indeed has $\binom{2n}{2}+2$ elements.

To show that we get all convex sets
$[i,\,j]$ with $-n+1\leq i\leq j\leq n-1,$
it will suffice to show that we get those of
the forms $[i,\,n-1]$ and $[-n+1,\,j],$
since $[i,\,j]=[i,\,n-1]\cap [-n+1,\,j].$
To get $[i,\,n-1],$ note that if $i<0,$ that set is the
union of the overlapping sets $[i,\,i+n-1]=A_{i+n}$ and
$[0,n-1]=A_n;$
if $i=0,$ it is $[0,\,n-1]=A_n,$ and if $i>0,$ it is a
difference of overlapping sets,
$[0,\,n-1]\setminus[-n+i,\,i-1]=A_n\setminus A_i.$
The symmetric argument shows that we can get
$[-n+1,\,j],$ completing the proof.
\end{proof}

Let us observe that for $n\leq 2,$ the members
of {\em every} family $\cC$ of $n$ subsets of a set $\Omega$
are convex under some ordering of $\Omega.$
This is immediate for $n=0$ and $n=1.$
To see the case $n=2,$ let $\cC=\{A_1,\,A_2\}.$
If $A_1$ and $A_2$ are disjoint or one contains
the other, it is again easy to see how to get such an ordering.
If they overlap, then by choosing any ordering in
which all elements of $A_1\setminus A_2$ lie above
those of $A_1\cap A_2,$ and these lie above
all elements of $A_2\setminus A_1,$ while every element not
in $A_1\cup A_2$ lies above or below all elements
of $A_1\cup A_2,$ we again get the desired orderability property.

For $n=0,\,1,\,2,$ the upper bounds given by
Proposition~\ref{P.size} are $2,$ $3$ and $9,$ while
those given by Proposition~\ref{P.size_cnvx} are $2,$ $3$ and $8.$
Thus, these functions differ first at $n=2.$
By the above orderability observation for $n=2,$
the upper bound given by Proposition~\ref{P.size_cnvx} supersedes
that given by Proposition~\ref{P.size}.

A consequence of Proposition~\ref{P.size_cnvx} is that given
a set $\cC$ of $n$ subsets of a finite set $\Omega,$ one can
determine in polynomial time whether there exists an ordering
of $\Omega$ making all elements of $\cC$ convex.
The idea is that starting with a list of the $n$ sets
in $\cC,$ together with the sets $\Omega$ and $\emptyset,$
one searches the list for overlapping pairs
$A,$ $B,$ and when one finds such a pair, adds to
the list their union,
their intersection, and their relative complements,
and repeats this process recursively.
Within polynomial time, one will know whether this recursive process
terminates (i.e., gives a family closed under those partial operations)
before the cardinality of the list goes above $\binom{2n}{2}+2.$
If it doesn't so terminate,
then by Proposition~\ref{P.size_cnvx}
there can be no ordering of the desired sort.
If it does, then one has an enumeration of $\cP,$
and one can work out its adjacency structure
and apply Theorem~\ref{T.cnvx}\,(a)\!$\iff$\!(b)
to determine whether there exists such an ordering.
The time needed will be polynomial in $n$ and the number of
elements in $\Omega.$
(If, rather, $\Omega$ is infinite, then whether one can do the same
in a time that is polynomial in $n$ will depend on what sort of
descriptions we have of $\Omega$ and the subsets
of $\Omega$ that form~$\cC.)$

\section{Other solutions to question~\eqref{d.FW_q}, using
results in the literature}\label{S.int_grph}

Given a set $\cC$ of subsets of a set $\Omega,$
the graph having the members of $\cC$ as vertices, and
an edge between $A,B\in\cC$ if and only if $A\cap B\neq\emptyset,$
is called the {\em intersection graph} of $\cC.$
There has been considerable study of the intersection graphs of finite
sets of intervals of the real line.
Graphs that can be so represented
are called {\em interval graphs}, and several
characterizations of such graphs have been established
(see~\cite{Wiki_int_graph}, and for some details,
\cite[Theorem~9.4.4]{graph} and~\cite{hsu}).

The literature varies, inter alia, as to whether intervals are assumed
to be closed or open \cite[Remark~9.4.2, point~2]{graph}.
In fact, it is not hard to show that a finite graph is
an interval graph in either sense if and only if it satisfies
the formally weaker condition that for {\em some}
totally ordered set $(\Omega,\leq),$ the graph can be represented
as the intersection graph of a set of convex subsets of $\Omega.$
However, criteria for a finite graph to be an interval graph
do not answer the finite-$\!\cC\!$ case of
question~\eqref{d.FW_q}:
If the intersection graph of $\cC$ is
an interval graph, this only tells us that it is isomorphic to
the intersection graph of {\em some} set of convex subsets
of some totally ordered set.
Observe, for example, that the complete graph with
three vertices is the intersection graph of the family of three sets
$\{1\}\subset\{1,2\}\subset\{1,2,3\},$ which are convex under the
standard ordering of the integers, but is also the
intersection graph of $\{x,y\},$ $\{y,z\},$ $\{x,z\},$
which are not simultaneously convex under any ordering of $\{x,y,z\}.$

However, Dave Witte Morris (personal communication)
has pointed out that one can get around this difficulty by
bringing in the singleton subsets of $\Omega$ alongside the elements
of $\cC$ (assuming for the moment that $\Omega$ is finite).
We develop the result below.
Here we take an {\em interval} to mean
a set $[s,t]=\{x\in\R\mid s\leq x\leq t\}$ where $s<t\in\R.$
(Because intervals are by definition nonempty, we shall, in
defining the set $\cC^+$ below, not only append to $\cC$ all singletons,
but also remove $\emptyset$ if it belonged to $\cC.)$

\begin{proposition}[D.\,W.\,Morris]\label{P.DWM}
Let $\Omega$ be a finite set, $\cC$ a set of subsets
of $\Omega,$ and $\cC^+$ the set
consisting of all members of $\cC$ other than $\emptyset,$
and also all singletons $\{a\}$ $(a\in\Omega).$
Then the following two conditions are equivalent:\\[.2em]
\textup{(i)}  There exists a total ordering $\leq$ of $\Omega$ under
which all members of $\cC$ are convex.\\[.2em]
\textup{(ii)}  The intersection graph of $\cC^+$ is an
interval graph \textup{(}i.e., is isomorphic
to the intersection graph of a set of
intervals of the real line\textup{)}.\vspace{.2em}

Thus, applying to $\cC^+$ any of
the criteria in~\cite{Wiki_int_graph} for a finite graph
to be an interval graph, one gets criteria for the
members of $\cC$ to be convex under some ordering of $\Omega.$
\end{proposition}

\begin{proof}
First assume we have an ordering $\leq$ as in~(i), and list
the elements of $\Omega$ in order as $a_1<\dots<a_n.$
Choose real numbers $s_1<t_1<s_2<t_2<\dots<s_n<t_n.$
These determine disjoint intervals $I_i=[s_i,t_i]$
$(i=1,\dots,n),$ ordered like the elements of $\Omega.$

To each nonempty convex subset $A$ of $\Omega,$ we can now associate
an interval $f(A)$ of the real line: the convex
closure of $\bigcup_{a_i\in A} I_i.$
It is straightforward to verify that for nonempty convex
subsets $A$ and $B$ of $\Omega,$ we have $A\cap B\neq\emptyset$ if
and only if $f(A)\cap f(B)\neq\emptyset.$
(The forward implication is immediate; conversely, if
$A$ and $B$ are disjoint, one can see from their
convexity that $f(A)$ and $f(B)$ will be disjoint.)
Applying $f$ to the members of $\cC^+,$ all of which
are convex in $(\Omega,\leq),$ we get the desired representation
of the intersection graph of $\cC^+$ as an interval graph.

Conversely, suppose the intersection graph of $\cC^+$ is
an interval graph,
i.e., that there exists a map $f$ carrying members of $\cC^+$
to intervals in the real line, such that
\begin{equation}\begin{minipage}[c]{35pc}\label{d.nonempty}
$f(A)\cap f(B)\neq\emptyset$ \quad if and only if \quad
$A\cap B\neq\emptyset.$
\end{minipage}\end{equation}
Since distinct singletons $\{a\}$ and $\{b\}$ $(a, b\in\Omega)$
are disjoint,~\eqref{d.nonempty} implies that
the $f(\{a\})$ $(a\in\Omega)$ will
be disjoint intervals, hence this family
of intervals has a natural ordering based on the ordering of $\R.$
In other words, there is a way of listing
the elements of $\Omega$ as $a_1,\dots,a_n$ so that, writing
$f(\{a_i\})=[s_i,t_i],$ we have $s_1<t_1<s_2<t_2<\dots<s_n<t_n.$
Hence let us order $\Omega$ by setting $a_1<\dots<a_n.$
For each $a\in\Omega$ and nonempty $A\in\cC,$ note that
\begin{equation}\begin{minipage}[c]{35pc}\label{d.in_iff}
$a\in A\ \iff\ \{a\}\cap A\neq\emptyset
\ \iff\ f(\{a\})\cap f(A)\neq\emptyset,$
\end{minipage}\end{equation}
where the second equivalence holds by~\eqref{d.nonempty}.
Now the convexity of $f(A),$
together with the fact that the $f(\{a_i\})$ $(i=1,\dots,n)$
form an ordered string of disjoint intervals, shows that
those $f(a_i)$ that are contained in $f(A)$ form a consecutive
substring of this string of intervals, whence by~\eqref{d.in_iff}
the set of $a\in A$ is a consecutive string under our
ordering of $\Omega;$ so every nonempty $A\in\cC$ is
indeed convex under that ordering.
$\cC$ might also contain $\emptyset,$ which we
had to exclude from $\cC^+;$ but $\emptyset$ is vacuously convex,
so our proof of~(i) is complete.
\end{proof}

Note that since we have been considering finite graphs only,
in the above result we had to take not only
$\cC,$ but also $\Omega$ to be finite; something I avoided
doing in earlier sections so that we could
make use of Lemma~\ref{L.inf<fin}.
However, given a not necessarily finite set $\Omega$ and a finite
set $\cC$ of subsets of $\Omega,$ we can proceed as follows.
For $a,b\in\Omega,$ let $a\sim b$ if and only
if the set of members of $\cC$ containing $a$ is the same
as the set of those containing $b.$
This is an equivalence relation on $\Omega$
with $\leq 2^{\r{card}(\cC)}$ equivalence classes.
Let $\Omega_0\subseteq\Omega$ be a set of
representatives of these finitely many equivalence classes,
and define $\cC_0=\{A\cap\Omega_0\mid A\in\cC\},$
a set of subsets of $\Omega_0.$
Then, on the one hand, if every member of $\cC$ is convex
under an ordering of $\Omega,$ every
member of $\cC_0$ is clearly convex
under the induced ordering of $\Omega_0.$
On the other hand, if there is an ordering $\leq_0$ of $\Omega_0$
under which all members of $\cC_0$ are convex,
we can construct an ordering $\leq$ of $\Omega$ by putting an
arbitrary internal total ordering on
each equivalence class under $\sim,$
and arranging those classes one above
another as their representatives are arranged in $\Omega_0.$
Then we see that all members of $\cC$ will be convex under the
resulting ordering.

Now Proposition~\ref{P.DWM} gives criteria for
the members of $\cC_0$ to be convex under an ordering of $\Omega_0,$
and by the above discussion this is equivalent to the existence of
an ordering of $\Omega$ under which the members of $\cC$ are convex.

Finally, we can apply Lemma~\ref{L.inf<fin} to get criteria
for such an ordering to exist in the case of infinite~$\cC.$

Two variants of the above approach were pointed out by
Martin Milani\v{c} (personal communication).

On the one hand, one may consider a family $\cC$ of subsets of
a set $\Omega$ as
determining a {\em bipartite} graph, where one family of vertices
corresponds to the elements of $\cC$ and the other
to the points of $\Omega,$ and an edge connects (the vertex
corresponding to) $S\in\cC$ with (the one corresponding to)
$p\in\Omega$ if and only if $p\in S.$
Results are known on when such a bipartite graph corresponds
to a family $\cC$ of convex sets under an ordering on
$\Omega$~\cite[\S\,9.7.2]{gr_clss}.

On the other hand, a pair $(\Omega,\,\cC)$ can be regarded
as a {\em hypergraph},
with $\Omega$ the set of vertices, and $\cC$ the set of hyperedges.
The question of when the hyperedges of a hypergraph are intervals
under some ordering of the vertex-set is studied
in~\cite{int_hyp}; see also~\cite[\S\,8.7]{gr_clss}.

\section{Acknowledgements}\label{S.Ackn}
I am indebted to Friedrich Wehrung for raising the
question~\eqref{d.FW_q} and for helpful comments on this write-up, and
to Dave Witte Morris and Martin Milani\v{c}
for the material of~\S\ref{S.int_grph}.

\end{document}